\documentclass{conm-p-l}

\usepackage{amssymb}

\usepackage{graphicx}

\usepackage{mathtools}

\usepackage{xcolor}

\newtheorem{theorem}{Theorem}[section]
\newtheorem{lemma}[theorem]{Lemma}

\theoremstyle{definition}
\newtheorem{definition}[theorem]{Definition}
\newtheorem{example}[theorem]{Example}

\newtheorem{question}[theorem]{Question}
\newtheorem{prop}[theorem]{Proposition}
\newtheorem{cor}[theorem]{Corollary}

\theoremstyle{remark}
\newtheorem{remark}[theorem]{Remark}

\newcommand{\Zhat}{\widehat{Z}}
\newcommand{\Zhathat}{\widehat{\vphantom{\rule{5pt}{10pt}}\smash{\widehat{Z}}\,}\!}

\numberwithin{equation}{section}

\begin{document}

\title{$(t,q)-$series invariants of Seifert manifolds}


\author{Louisa Liles}
\address{Department of Mathematics, Oberlin College, Oberlin OH 44074}
\email{lliles@oberlin.edu}
\thanks{The author acknowledges the partial support of NSF grant DMS-2105467, NSF RTG grant DMS-1839968, an AMS-Simons Travel Grant, and the Roselle-Huneke Award in Mathematics from the Jefferson Scholars Foundation.}


\date{}

\begin{abstract}
Gukov, Pei, Putrov, and Vafa developed a $q$-series invariant of negative definite plumbed $3$-manifolds with spin$^{c}$ structures, building on earlier work of Lawrence and Zagier. This series was recently generalized to an an infinite family of two-variable $(t,q)$-series invariants by Akhmechet, Johnson, and Krushkal. We calculate one such series for all Seifert manifolds with $b_{1}=0.$ This result extends a previous theorem of the author and McSpirit to any number of exceptional fibers, and the Reduction Theorem of Gukov, Katzarkov, and Svoboda to the two-variable case. As a consequence, we establish modularity properties and calculate radial limits for Seifert manifolds with three exceptional fibers. Expanding our analysis to the infinite collection of $(t,q)$-series, we give calculations for three infinite families of manifolds, finding mixed modularity properties for one such family.
\end{abstract}

\maketitle


\section{Introduction and statement of main results}

Lawrence and Zagier pioneered a $q$-series which defined a holomorphic function on the unit disk, unified the Witten-Reshetikhin-Turaev (WRT) invariants \cite{RT91, witten} of the Poincar\'e homology sphere, and was a key first example of a \textit{quantum modular form}, a term coined by Zagier with this series in mind \cite{lawrencezagier,zagier}. Gukov, Pei, Putrov, and Vafa extended this series to an invariant of negative definite plumbed $3$-manifolds with spin$^c$ structures, called $\Zhat$ \cite{gppv}. The $\Zhat$ invariant enjoys properties of Lawrence and Zagier's initial series, as it unifies the WRT invariants of all manifolds on which it is defined \cite{am_resurgence, fuji2021witten, gppv, gpp17, lawrencezagier, murakami}, and, for some infinite families of manifolds, remains a quantum modular form  \cite{bringmannquantum2, gukovmodularity, cheng2024quantum, hikami, lawrencezagier}. 

Recently the $\Zhat$ series was generalized to an infinite collection of $(t,q)$-series invariants \cite{ajk}, giving a common refinement of $\Zhat$  and N\'emethi's theory of lattice homology \cite{nemethi}. These objects differ from the two-variable series invariants of knot complements presented in \cite{gukovmanolescu}. However both invariants may be obtained from the three-variable series for knot complements developed in \cite{ajp}, using tools from knot lattice homology. The $(t,q)$-series of \cite{ajk} are recovered by a surgery formula, whereas the knot complement invariants of \cite{gukovmanolescu} are recovered by a specialization of the third variable. This paper will focus on the infinite collection of $(t,q)$-series invariants defined in \cite{ajk}.

One invariant in this collection, called $\Zhathat(t,q),$ recovers the $\Zhat$ series when $t=1.$ The $\Zhathat$ series was calculated for Brieskorn homology spheres in \cite[Theorem 3.1]{LM}. On the other hand, \cite[Theorem 4.1]{GKS} calculated $\Zhat$ for Seifert manifolds with $b_{1}=0$. 
The following result extends the calculation of \cite{LM} to Seifert manifolds with any number of exceptional fibers, and that of \cite{GKS} to the two-variable setting. Note all manifolds are assumed to be negative definite (see section \ref{subsec: negdef}) with $b_{1}=0$. 
\begin{theorem}\label{thm: nstar main}  Let $M(b;(a_{1},b_{1}),\dots,(a_{k},b_{k}))$ be a Seifert manifold, and let $H=H_{1}(M;\mathbb{Z})$. Then:
    \[\Zhathat_{0}(M)(t^{2},q^{|H|})=q^{\Lambda}\mathcal{L}_{A}(\text{s.e.}(f_{0}(t,z))).\] \end{theorem} 
The full statement appears as Theorem \ref{thm: nstar main full}, which includes calculations of the spin$^{c}$ dependent series $\Zhathat_{\sigma}(t,q)$. Although some definitions are postponed until then, Theorem \ref{thm: nstar main} establishes
$\Zhathat_{0}(M)(t^{2},q^{|H|})$ as the product of a rational power of $\Lambda$ of $q$ and a Laplace transform $\mathcal{L}_{A}$ of a formal power series given by a \textit{symmetric expansion} (as in \cite{GKS}) of a rational function $f_{0}(t,z)$. We will continue letting $\Zhathat(M)(t,q)$ denote the invariant associated to a manifold $M$, but when the manifold is clear from context we will omit $M$ from the expression.
\begin{remark} It was shown in \cite[Equation (4.8)]{GKS} that $\Lambda=6\lambda(M)+\gamma$, where $\lambda(M)$ is the Casson-Walker invariant and $\gamma$ is determined by the Seifert data.
\end{remark}

The calculation in Theorem \ref{thm: nstar main} allows us to establish modularity properties for Seifert manifolds with three exceptional fibers.
\begin{theorem}\label{thm: nstar modularity} Let $M=M(b;(a_1,b_1),(a_2,b_2),(a_3,b_3))$ and fix $\omega$ to be a $2j^{\text{th}}$ root of unity. Then ${\mathcal{L}}_{A}(\text{s.e.}(f_{0}(\omega,z)))$ is, up to normalization, a quantum modular form of weight $\frac{1}{2}$ with respect to the subgroup $\Gamma(4a_{1}a_{2}a_{3}j^{2})\subset \text{SL}(2,\mathbb{Z}).$ 

\end{theorem} 

Our proof of Theorem \ref{thm: nstar modularity} implies the existence of radial limits for Seifert manifolds with three exceptional fibers. For $\omega$ as above and any root of unity $\xi$, limits of $\Zhathat(\omega^2,q)$ as $q \nearrow \xi$ exist and can be expressed as finite sums, see Remark \ref{rem: radial3}.

The $\Zhathat(t,q)$ invariant is one instance in an infinite family of two-variable series invariants developed in \cite{ajk}. This collection, which we denote $P_{W}^{\infty}(t,q)$, is indexed by \textit{admissible families} of functions $W$ \cite[Section $4$]{ajk}. For Seifert manifolds with three exceptional fibers, the role of $W$ is illuminated by the following theorem: 
\begin{theorem}\label{thm: 3star main}
    Fix an admissible family $W$ and a Seifert manifold \\$M(b;(a_1,b_1),(a_2,b_2),(a_3,b_3)).$ Let $H=H_{1}(M;\mathbb{Z})$. Then \begin{equation}\label{eq: blah} P^{\infty}_{W,[0]}(M)(t^{2},q^{|H|})=q^{\Lambda}(\mathcal{L}_{A}(\text{a.e.}_{W_{3}(1)}(f_0(t,z))).\end{equation}
     \end{theorem}  Here, $\text{a.e.}_{W_{3}(1)}$ denotes an \textit{asymmetric expansion} (see Definition \ref{def: ae}). 
 The full statement of this result appears as Theorem \ref{thm: 3star series calc full}.    
     \begin{remark} Setting $W_3(1)=\frac{1}{2}$ recovers the $\Zhathat$ invariant.\end{remark}
From the calculation in Theorem \ref{thm: 3star main} we obtain a mixed modularity result when $t$ is any root of unity:

\begin{theorem}\label{thm: 3star modularity}
    Fix a $3$-manifold $M(b;(a_1,b_1),(a_2,b_2),(a_3,b_3))$, a $2j$-th root of unity $\omega$, and an admissible family $W$. The one-variable series $\mathcal{L}_{A}(\text{a.e.}_{W_{3}}(f_{0}(\omega,z))$ is, up to normalization, a sum of quantum-modular and modular forms of weight $\frac{1}{2}$ with respect to $\Gamma(4a_{1}a_{2}a_{3}j^{2})\subset \text{SL}(2,\mathbb{Z}).$ Specifically, \[\mathcal{L}_{A}(\text{a.e.}_{W_{3}(1)}(f_{0}(\omega,z)))=p(q)+\theta_{f}+(\frac{1}{2}-W_{3}(1))\Theta_{f},\] where $p(q)$ is a polynomial, $\theta_{f}$ is a quantum modular form, and $\Theta_{f}$ is a modular form. 
\end{theorem}

\begin{remark} Setting $W_{3}(1)=\frac{1}{2}$ recovers the result of Theorem \ref{thm: nstar modularity}. The novelty of Theorem \ref{thm: 3star modularity} is that any choice of admissible family $W$ yields (mixed) modularity properties, but quantum modularity on its own only occurs with $\Zhathat$.

In Section \ref{sec: series calc}, the admissible family $W$ is shown to play a similar role in ``asymmetrizing" the two-variable series associated to $H$-shaped plumbing graphs and plumbing graphs with one node and four leaves. Explicit formulae are shown in Examples \ref{ex: Hshape} and \ref{ex: 4star}, however the modularity properties of these series are not presently known by the author. Modularity properties of the one-variable $\Zhat$ invariants for these families are established in \cite{bringmannquantum1, bringmannquantum2}. \end{remark}

\begin{question} Is there an analogue of Theorem \ref{thm: 3star modularity} for other families of manifolds? Do certain choices of $W$ and $t$ give rise to modularity properties?

\end{question}

\begin{question}
    It follows from Theorem \ref{thm: 3star main} and the definition of \textit{asymmetric expansions} in Section \ref{sec: series calc} that admissible families with $W_{3}(1)=1$ or $W_{3}(1)=0$ create edge cases for $M=M(b;(a_{1},b_{1}),(a_{2},b_{2}),(a_{3},b_{3}))$, in which only one Laurent expansion is included in the asymmetric expansion. It could be interesting to investigate the properties and potential topological interpretations of these two specializations, which may be easier to compute.
\end{question}

\begin{question} Modularity properties for the spin$^c$ dependent series $\Zhat_{\sigma}$ are established in \cite{gukovmodularity, cheng20243d, cheng2024quantum}. Do similar properties hold for $\Zhathat_{\sigma}(t,q)$ when $t$ is fixed to be a root of unity? \end{question}

\section*{Acknowledgments} The author is grateful for the guidance of her advisor Vyacheslav Krushkal and for thoughtful commentary from Eleanor McSpirit. The author also thanks Sergei Gukov, Josef Svoboda, and Lara San Mart\'in Su\'arez for insightful conversations and suggestions. This project was inspired by a discussion with Josef Svoboda about his joint work with Sergei Gukov and Ludmil Katzarkov on  $\Zhat$ invariants and splice diagrams for Seifert manifolds \cite{GKS}. The author thanks the referees and Marco Aldi, Allison Moore, and Niocla Tarasca, the organizers of the Richmond Geometry Meeting.


\section{Background}\label{sec: bg}
\subsection{Negative definite plumbed $3$-manifolds}\label{subsec: negdef} 
A plumbed $3$-manifold $M$ is described by a \textit{plumbing graph} $\Gamma$ with integer-weighted vertices. We restrict to the cases in which $\Gamma$ is a tree. The plumbing graph $\Gamma$ has an associated framed link $L(\Gamma) \hookrightarrow S^{3}$ on which one can perform Dehn surgery to obtain a $3$-manifold $M(\Gamma)$. This framed link consists of an unknotted component for each vertex, whose framing is given by the vertex's weight. Edges of $\Gamma$ correspond to clasps between unknots, as in Figure \ref{fig: dehn}.
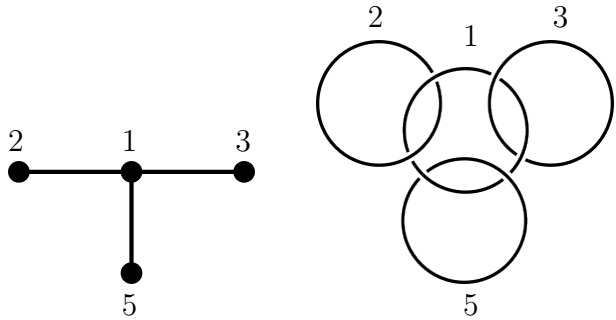
\begin{figure}    \includegraphics[scale=0.5]{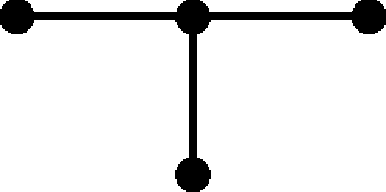}\put(-50,50){$1$}\put(-93,50){$2$}\put(-7,50){$3$}\put(-50,-12
){$5$}\qquad\includegraphics[scale=0.5]{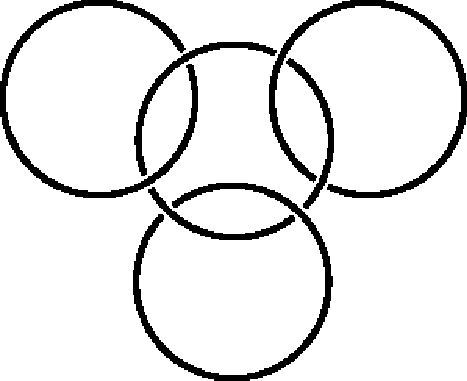}\put(-57,90){$1$}\put(-93,97){$2$}\put(-23,97){$3$}\put(-57,-12){$5$}\caption{A plumbing graph and its associated framed link.}\label{fig: dehn}
\end{figure}

Let $s$ denote the number of vertices of $\Gamma$. Choosing an order on the vertices, define $\delta \in \mathbb{Z}^{s}$ so that $\delta_{i}$ is the degree of the $i$th vertex, and $m \in \mathbb{Z}^{s}$ so that $m_{i}$ is the integer weight assigned to the $i$th vertex $v_{i}$. The $s \times s$ \textit{plumbing matrix} $M$ for $\Gamma$ is given by:
\[M_{i,j}=\begin{cases}m_{i} & i=j \\ 1 & i \neq j\text{ and $v_{i}$ and $v_{j}$ are connected by an edge} \\ 0 & \text{otherwise.}\end{cases}\]Note $M$ denotes both the plumbing matrix and the $3$-manifold it represents. A general $3$-manifold $Y$ is called \textit{negative definite plumbed} if it is homeomorphic to some $M(\Gamma)$, with $M$ a negative definite matrix. Different graphs may represent homeomorphic negative definite plumbed manifolds; in fact this is the case if and only if they are related by a sequence of \textit{Neumann moves} of type (a) and (b) \cite{neumann}.

We follow the conventions of \cite{GKS} when describing and ordering vertices of $\Gamma$. Specifically, vertices of degree $1$ are called \textit{leaves}, vertices of degree $2$ are called \textit{joints}, and vertices of degree $\geq3$ are called \textit{nodes.} We will always order vertices so that nodes are first, followed by leaves, followed by joints. 

\subsubsection{Seifert manifolds}
To a Seifert manifold $M=M(b; (a_{1},b_{1}),\dots,(a_{k},b_{k}))$ with $0< b_i < a_i$ and $\gcd{(a_i,b_i)}=1$ for all $1 \leq i \leq k$, we associate a plumbing graph with one node of degree $k$ and weight $-b$, $k$ leaves, and some joints connecting the leaves to the node. The $k$ legs of the star-shaped graph are given by the expansions of $\frac{a_{i}}{b_{i}}, 1 \leq i \leq k,$ as continued fractions (see Figure \ref{fig: seifert}.)
    \[\frac{a_{i}}{b_{i}}=c_{1}^{i} - \cfrac{1}{c_{2}^{i}-\cfrac{1}{\ddots -\cfrac{1}{c_{\ell_{i}}^{i}}}}\] 

The Reduction Theorem of \cite{GKS} gave a formula for the $\Zhat$ series of Seifert manifolds that, up to the prefactor $q^{\Lambda}$, depends only on $a_{1},\dots, a_{k}$. This is also the case for Theorem \ref{thm: nstar main full} in Section \ref{sec: nstar main}. 
\vspace{0.5cm}
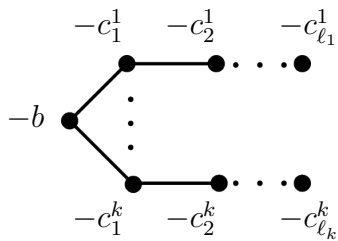
\begin{figure}\includegraphics[scale=0.4]{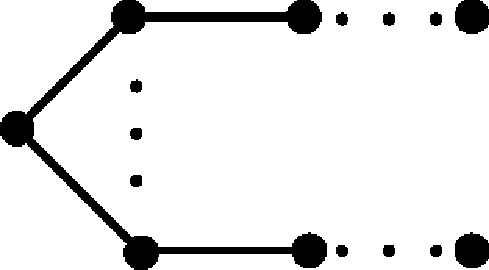}\put(-115,25){$-b$}\put(-90,60){$-c^{1}_{1}$}\put(-55,60){$-c^{1}_{2}$}\put(-12,60){$-c^{1}_{\ell_{1}}$}\put(-12,-13){$-c^{k}_{\ell_{k}}$}\put(-55,-13){$-c^{k}_{2}$}\put(-90,-13){$-c^{k}_{1}$}\caption{A plumbing tree for $M(b;(a_{1},b_{1}),\dots,(a_{k},b_{k}))$ obtained from continued fraction decompositions of $\frac{a_{i}}{b_{i}}$.}\label{fig: seifert}\end{figure}

\subsection{Spin$^{c}$ structures} The $q$ and $(t,q)$-series invariants discussed in this paper are topological invariants of negative definite plumbed $3$-manifolds with spin$^{c}$ structures. For our purposes it suffices to think of spin$^{c}$ structures as \[\text{spin}^{c}(M)\cong \frac{m+2\mathbb{Z}^{s}}{2M\mathbb{Z}^{s}}\cong \frac{\delta+2\mathbb{Z}^{s}}{2M\mathbb{Z}^{s}},\] where the second map sends $[k] \mapsto [k -(m +\delta)]=[k-Mu]$, where $u:=(1,1,1,\dots).$ More details on spin$^{c}$ structures can be found in \cite[Section $2.2$]{GKS}. Note that the contour integral definition of $\Zhat$ appearing in \cite{gppv, gpp17, GKS} takes $[a] \in \frac{\delta+2\mathbb{Z}^{s}}{2M\mathbb{Z}^{s}}$ as an input, whereas the combinatorial definition of $\Zhathat$ in \cite{ajk} involves $[k] \in \frac{m+2\mathbb{Z}^{s}}{2M\mathbb{Z}^{s}}.$ 

Following the techniques and notational conventions of \cite{GKS}, we will use the action of the first homology group on spin$^{c}$ structures $\sigma \in \frac{\delta + 2\mathbb{Z}^s}{2M\mathbb{Z}^s}$ to understand the series $\Zhathat_{\sigma}$. We let $H$ denote the group $H_1(M;\mathbb{Z})\cong\mathbb{Z}^{s}\slash M\mathbb{Z}^s$, and we let $\sigma^{ac}$ denote the ``anticanonical" spin$^{c}$ structure $\delta-2u$. The action of $h \in H$ on $\sigma^{ac}$ is given by $h\sigma^{ac}:=\sigma^{ac}+2h$. Letting $g_i:=[e_i]$, the equivalence class in $H$ of the elementary standard basis vector $e_i \in \mathbb{Z}^s$, we get a presentation for $H$: \[\langle g_0,\dots,g_k| [g_i,g_0]=1,g_i^{a_i}=g_0, \prod_{1 \leq i \leq k}g_i^{b_i}=g_0^{b} \rangle, \] where $a_i$, $b_i$, and $b$ come from the Seifert data, $g_0$ corresponds to the node, and the $g_i$'s correspond to the $k$ leaves.

\subsection{The $(t,q)$-series invariants}

The authors of \cite{ajk} develop an infinite collection of $(t,q)$-series invariants as a common refinement of $\Zhat$ and N\'emethi's theory of lattice homology \cite{nemethi}. These series, denoted $P_{W,[k]}^{\infty}(t,q)$, depend on a choice of \textit{admissible family} $W$  of functions \[\{W_{n}:\mathbb{Z} \to \mathcal{R}\}_{n \in \mathbb{N}}\] for some commutative ring $\mathcal{R}$; axioms for admissibility are given in \cite[Definition $4.1$]{ajk}). Fixing $\mathcal{R}$, there are infinitely many admissible families giving rise to infinitely many two-variable series invariants \cite[Proposition $4.4$]{ajk}. In this paper $\mathcal{R}=\frac{\mathbb{Z}}{2}.$  The notation for admissible families used in \cite{ajk} is $F$, rather than $W$. Our use of $W$ is consistent with that of \cite{ajp} who used admissible families to define three-variable series invariants of plumbed knot complements. The authors made this change to avoid confusion with the related $F_{K}$ invariants defined earlier in \cite{gukovmanolescu}. 
\begin{definition}[\cite{ajk}]
Fix a negative definite plumbed $3$-manifold $M$ given by a negative definite plumbing graph $\Gamma$, an admissible family $W$, and a spin$^{c}$ representative $k$ for $\displaystyle{[k] \in \frac{m+2\mathbb{Z}^{s}}{2M\mathbb{Z}^{s}}}$. Let $v(\Gamma)$ denote the set of vertices in $\Gamma$, let $\langle-,-\rangle$ denote the bilinear form given by the plumbing matrix $M$, and let $(\cdot)$ denote the standard Euclidean inner product. For $x \in \mathbb{Z}^{s}$, we define 
\[\chi_{k}(x):=\frac{-k \cdot x + \langle x, x \rangle}{2} \in \mathbb{Z}\] and \[W_{\Gamma,k}(x):=\prod_{v_{i} \in v(\Gamma)}W_{\delta_{i}}((2Mx+k-Mu)_{i}).\] We also define 

\[\Theta_{k}:=\frac{k \cdot u - \langle u, u \rangle}{2}\] and \[\varepsilon_{k}(x):=-\frac{(k-Mu)^{2}+3s+\sum_{v}m_{v}}{4}+2\chi_{k}(x)+\langle x, u \rangle.\]
\end{definition}
\begin{theorem}[\cite{ajk}, Theorem $6.3$ and Remark $6.5$] Fix an admissible family $W$. Given a negative definite plumbed $3$-manifold $M$ equipped with a spin$^{c}$ structure $[k]$, the series \[P_{W,[k]}^{\infty}(M)(t,q):=\sum_{x \in \mathbb{Z}^{s}}W_{\Gamma,k}(x)q^{\varepsilon_{k}(x)}t^{\Theta_{k}+\langle x, u \rangle}\] is a topological invariant of the pair $(M, [k]).$
\end{theorem} A particular family $W=\widehat{W}$ gives rise to the series \[\Zhathat_{[k]}(t,q):=P_{\widehat{W},[k]}^{\infty},\] for which \[\Zhathat_{[k]}(1,q)=\Zhat_{[k-Mu]}(q).\]
The family $\widehat{W}$ is constructed in \cite[Definition $7.1$]{ajk} so that $\widehat{W}_{d}(n)$ captures the coefficient on $z^{-n}$ in the \textit{symmetric expansion} of $(z-z^{-1})^{2-d}$.

\begin{definition}\label{def: se} Let $f$ be a holomorphic function with Laurent expansions centered at $z=0$ and $z=\infty$, given by \[\sum_{k}a_{k}z^{k} \text{ and }\sum_{k}b_{k}z^{-k}.\] The symmetric expansion, denoted $\text{s.e.}(f),$ is the formal power series  \[\text{s.e.}(f):=\frac{1}{2}\sum_{k}a_{k}z^{k} +\frac{1}{2}\sum_{k}b_{k}z^{-k}.\]
\end{definition} We will use the following fact about $\text{s.e.}((z-z^{-1})^{2-d})$ throughout the paper.

\begin{lemma}\label{lem: coeff} The coefficient on $z^{-n}$ is nonzero in $\text{s.e.}((z-z^{-1})^{2-d})$ only if $n\equiv d\pmod{2}.$
\end{lemma}
\begin{proof}
The conclusion follows from the fact that the Laurent expansion of $(z-z^{-1})^{2-d}$ on $|z|<1$ is \[(-\sum_{i\geq 0}z^{2i+1})^{d-2},\] and the expansion on $|z|>1$ is \[(\sum_{i\geq 0}z^{-(2i+1)})^{d-2}.\]
\end{proof}
One can remove the spin$^{c}$ dependence of $\Zhat$ by instead studying the sum of the series associated to each spin$^{c}$ structure, denoted $\Zhat_{0}(q).$ We now introduce the analogue for the two variable series.

\begin{definition}
    Let $M$ be a negative definite plumbed $3$-manifold. Define \[P^{\infty}_{W,0}(M)(t,q):=\sum_{[k] \in \text{spin}^{c}(M)}P^{\infty}_{W,[k]}(M)(t,q).\] 
\end{definition}

The following will be used to calculate $P_{W,0}^{\infty}(t^{2},q)$ throughout the paper. 
\begin{lemma}\label{lem: shortcut} Let $M$ be a negative definite plumbed $3$-manifold. Define $m$, $s$, and $u$ as before, and let
\[\eta:=-\frac{3s+m\cdot u}{4}.\] Then

\[P_{W,0}^{\infty}(M)(t^{2},q)=q^\eta\sum_{\ell \in \delta + 2\mathbb{Z}^{s}}(\prod_{i}W_{\delta_{i}}(\ell_{i})))q^{-\frac{\ell^{t}M^{-1}\ell}{4}}t^{\sum_{i}\ell_{i}}.\]
\end{lemma}
\begin{proof} From the definition in \cite{ajk} and a simple calculation in \cite[Section $3$]{LM}, we know that for a given spin$^{c}$ structure $[k]$, 
\[P_{W,[k]}^{\infty}(t^{2},q)=q^\eta\sum_{x \in \mathbb{Z}^{s}}W_{\Gamma,k}(x)q^{-\frac{\ell^{t}M^{-1}\ell}{4}}t^{2(\Theta_{k}+\langle x, u \rangle)},\] where $\ell=k-Mu+2Mx$. A straightforward calculation shows that \[2(\Theta_{k} + \langle x, u \rangle )= \ell \cdot u=\sum_{i}\ell_{i}.\]

Using the fact that $W_{\Gamma,k}(x)=\prod_{i}W_{\delta_{i}}(k-Mu+2Mx)_{i}$, we may re-index the sum as
\[P_{W,[k]}^{\infty}(t^{2},q)=q^\eta\sum_{\ell \in k-Mu+2M\mathbb{Z}^{s}}\prod_{i}(W_{\delta_{i}}(\ell_{i}))q^{-\frac{\ell^{t}M^{-1}\ell}{4}}t^{\ell \cdot u}.\]
 
Now summing over all $[k]\in\frac{m+2\mathbb{Z}^{s}}{2M\mathbb{Z}^{s}}$, we may write 
 
\begin{align*}P_{W,0}^{\infty}(t^{2},q)&=\sum_{[k] \in \text{spin}^{c}(M)}P_{W,[k]}^{\infty}(M)\\
&=\sum_{[k]\in \frac{m + 2\mathbb{Z}^{s}}{2M\mathbb{Z}^{s}}}q^\eta\sum_{\ell \in k-Mu+2M\mathbb{Z}^{s}}\prod_{i}(W_{\delta_{i}}(\ell_{i}))q^{-\frac{\ell^{t}M^{-1}\ell}{4}}t^{\ell \cdot u}\\
&=q^\eta\sum_{\delta + 2\mathbb{Z}^{s}}\prod_{i}(W_{\delta_{i}}(\ell_{i}))q^{-\frac{\ell^{t}M^{-1}\ell}{4}}t^{\ell \cdot u}.
\end{align*}

The previous equality is true because for each $[k] \in \frac{m+2\mathbb{Z}^{s}}{2M\mathbb{Z}^{s}}$, we are summing over all possible representatives of its image in $\frac{\delta +2\mathbb{Z}^s}{2M\mathbb{Z}^s}$ under the isomorphism sending $[k] \mapsto [k-Mu]$.
This is equivalent to taking the sum over all $\delta+2\mathbb{Z}^{s}.$

\end{proof}

\begin{remark} We may further simplify the formula above by considering only those $\ell$ for which $\ell_i=0$ when $\delta_i=2$, and $\ell_i=\pm1$ when $\delta_i=1$. This is because, by the definition of admissible families in \cite{ajk}, $W_{2}(\ell_i)=0$ unless $\ell_i=0$ and $W_{1}(\ell_i)=0$ unless $\ell_i=\pm 1$. 
\end{remark}
\begin{remark}\label{rem: se prod} For a Seifert manifold $M=M(b;(a_1,b_1),\dots,(a_k,b_k))$, the quantity \[\prod_{i}(\widehat{W}_{\delta_{i}}(\ell_{i}))\] from Lemma \ref{lem: shortcut} describes the coefficient of $z_{1}^{-\ell_{1}}\dots z_{s}^{-\ell_{s}}$ in \[\text{s.e.}(\prod_{v}(z_{v}-z_{v}^{-1})^{2-\delta_{v}}),\] where one expansion is on $|z_{0}|<1$ and the other is on $|z_{0}|>1$; observe that 
\[\text{s.e.}(\prod_{v}(z_{v}-z_{v}^{-1})^{2-\delta_{v}}))=\left[\prod_{1\leq i \leq k}(z_{v_{i}}-z_{v_{i}}^{-1})\right]\text{s.e.}\left[(z_{v_{0}}-z_{v_{0}}^{-1})^{2-k}\right],\] i.e. the symmetric expansion of the product is equal to the product of symmetric expansions. This holds for any manifold whose plumbing graph $\Gamma$ is ``star-shaped" with one node and $k$ leaves. 
\end{remark}

Throughout this paper we may write $(z_1,\dots,z_s)$ as $\vec{z}$ and $z_{1}^{\ell_{1}}\dots z_{s}^{\ell_{s}}$ as $\vec{z}^{\ell}$.

\subsection{Modular and quantum modular forms} This section provides a brief overview of modular and quantum modular forms of half-integral weight. For more details, we direct the reader to \cite{web, shimura}.

Let $\mathbb{H}$ denote the upper half complex plane. 
Modular forms can be thought of as holomorphic functions $f: \mathbb{H} \to \mathbb{C}$ that behave ``nicely" with respect to the action of $\text{SL}(2,\mathbb{Z})$ on $\mathbb{H}$ by linear fractional transformations. The meaning of ``nicely" is specified in Definition \ref{def: mod} below. For $\gamma = \begin{psmallmatrix} a & b \\ c & d \end{psmallmatrix} \in \text{SL}(2,\mathbb{Z})$, its action on $\tau \in \mathbb{H}$ is 
\[\gamma\tau := \frac{a \tau + b}{c \tau +d}.\]

Often functions are not modular with respect to the action of $\text{SL}(2,\mathbb{Z}),$ but rather with respect to some subgroup. Two subgroups that appear in the theory of half-integral weight modular forms are

\begin{align*}
\Gamma_1(N)&:= \left\{\begin{pmatrix} a & b \\ c & d
\end{pmatrix} \in \text{SL}(2,\mathbb{Z}) : \ a \equiv d \equiv 1 \bmod{N}, \ c \equiv 0 \bmod{N}\right\}; \\ \Gamma(N)&:= \left\{\begin{pmatrix} a & b \\ c & d
\end{pmatrix} \in \text{SL}(2,\mathbb{Z}) : \ a \equiv d \equiv 1 \bmod{N}, \ b \equiv c \equiv 0 \bmod{N}\right\},
\end{align*} where $N$ is called the \textit{level} of the subgroup. We also define, for $d$ odd, 

\[\varepsilon_d := \begin{cases}
1 & \text{if } d \equiv 1 \bmod{4}; \\
i & \text{if } d \equiv 3 \bmod{4},
\end{cases}
\] and for $k \in\frac{\mathbb{Z}}{2},$ the Petersson slash operator of weight $k$
\[
f|_k\gamma(\tau):= \begin{cases}
(c\tau + d)^{-k} f(\gamma \tau)& \text{ if } k \in \mathbb{Z}; \\
\varepsilon_d^{2k}\left(\frac{c}{d}\right)(c\tau + d)^{-k}f(\gamma \tau) & \text{ if } k \in \frac{1}{2} + \mathbb{Z},
\end{cases}\] where $(\frac{\cdot}{\cdot})$ denotes the Jacobi symbol.

\begin{definition}\label{def: mod} Let $\Gamma \leq \text{SL}(2,\mathbb{Z})$ be of level $N$ with $4 | N.$ A holomorphic function $f: \mathbb{H} \to \mathbb{C}$ is a \textit{modular (resp. cusp) form} with multiplier $\chi$ and weight $k \in \frac{\mathbb{Z}}{2}$ if 
\begin{enumerate}
\item for all $\gamma \in \Gamma$, $f -\overline{\chi}(\gamma) f|_{k}\gamma=0$, and
\item for all $\gamma \in \emph{SL}(2,\mathbb{Z})$, $(c\tau + d)^{-k}f( \gamma \tau)$ is bounded (resp. vanishes) as $\tau \to i\infty$. 
\end{enumerate}
\end{definition}

Observe that $\text{SL}(2,\mathbb{Z})$ acts transitively on $\mathbb{Q} \cup i\infty.$ Therefore if one wanted a function $f: \mathbb{Q} \to \mathbb{C}$ to satisfy the transformation law $(1)$ from the above definition, $f$ would be completely determined by a single pair $(x,f(x)).$ A more interesting object $f: \mathbb{Q} \to \mathbb{C}$ is a \textit{quantum modular form}, which admits some reasonably ``nice" obstruction to modularity. Quantum modular forms of half-integral weight are defined as follows. 

\begin{definition} Let $\Gamma$ have level $4 | N$, and let $\mathcal{Q}$ denote $\mathbb{Q}\backslash S$ for some discrete set $S$ which is closed under the action of $\Gamma$. We define a \textit{quantum modular form} of weight $k$ with multiplier $\chi$ for $\Gamma$ to be a function $f\colon \mathcal{Q} \to \mathbb{C}$ such that for all $\gamma = \begin{psmallmatrix} a & b \\ c & d
\end{psmallmatrix} \in \Gamma$, the functions $h_\gamma\colon \mathcal{Q}\backslash \{\gamma^{-1}(i \infty)\}) \to \mathbb{C}$,
\begin{equation*}h_\gamma(x) := f(x) - \overline{\chi}(\gamma)f|_k \gamma(x)\end{equation*}
extend to some ``nice" function on $\mathbb{R}$. 
\end{definition}
\begin{remark}
    The definition of quantum modularity is still under construction, and ``nice" takes on different meanings in different contexts. The obstructions to modularity in Theorems \ref{thm: nstar modularity} and \ref{thm: 3star modularity} extend to functions that are $\mathcal{C}^{\infty}$ on $\mathbb{R}$ and real-analytic on $\mathbb{R}\setminus\gamma^{-1}(i\infty).$
\end{remark}

Quantum modularity properties of $q$-series invariants have since been established for several infinite families of manifolds \cite{bringmannquantum1,bringmannquantum2}. The $(t,q)$ series $\Zhathat$ has been shown to produce infinitely many quantum modular invariants of Brieskorn spheres via specialization of the $t$ variable to a root of unity $\zeta$ \cite[Theorem $1.2$]{LM}. These $q$-series have radial limits toward roots of unity that can be thought of as $\zeta$-deformations of WRT invariants for Brieskorn spheres \cite[Theorem $1.1$]{LM}. In the notation of this paper, $\zeta=\omega^{2}.$

\section{Calculation of $\Zhathat$ for Seifert manifolds}\label{sec: nstar main} This section develops a formula for the $\Zhathat_{0}(t,q)$ and $\Zhathat_{\sigma}(t,q)$ series invariants of Seifert manifolds. We begin with some lemmas about the relevant symmetric expansions.

\begin{lemma}
    Let $M=M(b;(a_1,b_1),\dots,(a_k,b_k))$. For all $\ell \in \delta + 2\mathbb{Z}^{s}$, let $C(\ell)$ denote the coefficient of ${z}^{{\ell}}$ in \[\text{s.e.}(\prod_{v}(z_{v}-z_{v}^{-1})^{2-\delta_{v}}).\] Then \[C(\ell)=C(-\ell).\]
\end{lemma} 
\begin{proof}
    {}Sending $\ell \mapsto -\ell$ multiplies the coefficient by $(-1)^{k}(-1)^{k-2}=1$. The factor of $(-1)^{k}$ comes from sending $\ell_{i}\mapsto -\ell_{i}$ for $1 \leq i \leq k$, and the factor of $(-1)^{k-2}$ comes from sending $\ell_{0}\mapsto -\ell_{0}$.
\end{proof}
\begin{lemma} Given a Seifert manifold $M(b;(a_1,b_1),\dots, (a_k,b_k))$, let $A:=\prod_{i}a_{i}$ and  $\overline{a_{i}}:=\frac{A}{a_{i}}$. Define
    \[f_{0}(z,t):=\frac{(z^{\overline{a_{1}}}t^{-1}-z^{-\overline{a_{1}}}t)\dots((z^{\overline{a_{k}}}t^{-1}-z^{-\overline{a_{k}}}t)}{(z^{A}t^{-1}-z^{-A}t)^{k-2}},\] and define \[\text{s.e.}(f_0(t,z))\] to be the formal power series obtained by taking the average of the expansions on $|z^{A}t^{-1}|<1$ and $|z^{A}t^{-1}|>1$, where in both cases $z^{A}t^{-1}$ is treated as a single variable. Furthermore, let 
    $C_0(\ell)$ denote the coefficient of $z^{\ell_{0}A+\sum_{i\geq 1}\overline{a_{i}}\ell_{i}}t^{-\ell\cdot u}$ in \[\text{s.e.}(f_{0}(z,t)).\] Then \[C(\ell)=C_0(\ell).\]
    
\end{lemma}

\begin{proof} This follows from the fact that $f_{0}(z,t)$ can be obtained from \[\prod_{v}(z_{v}-z_{v}^{-1})^{2-\delta_{v}}\] 
 via the substitution \begin{align*}
    z_{0}&\mapsto z^{A}t^{-1}\\
    z_{i}&\mapsto z^{\overline{a_{i}}}t^{-1}\qquad (1 \leq i \leq k).
\end{align*} 
    
\end{proof}

\begin{cor}\label{cor: cl}
$C_0(\ell)=C_0(-\ell).$
\end{cor}

Toward an understanding of the spin$^{c}$-dependent $(t,q)$-series, we extend the projection map $\pi_{\sigma}$ from \cite{GKS}.
Define \[F_0(t,\vec{z}):=\prod_v2\text{s.e.}(z_vt^{-1}-z_v^{-1}t)^{2-\delta_v},\]where in each of the symmetric expansions above we treat $z_vt^{-1}$ as a single variable. Note that $F_0(1,\vec{z})=F_0(\vec{z})$ in the notation of \cite{GKS}. For a spin$^{c}$ structure $\sigma \in \frac{\delta + 2\mathbb{Z}^s}{2M\mathbb{Z}^s}$, let \[F_{\sigma}(t,\vec{z}):=\pi_{\sigma}\big(\prod_v2\text{s.e.}(z_vt^{-1}-z_v^{-1}t)^{2-\delta_v}\big),\] where  \[\pi_{\sigma}(\vec{z}^{\ell}t^{\ell\cdot u}):=\begin{cases} \vec{z}^{\ell}t^{\ell\cdot u} & [\ell]=\sigma\\ 0 & \text{otherwise.} \end{cases}\] 

The authors of \cite{GKS} use the action of $H$ on spin$^{c}$ structures to understand the image of $\pi_{\sigma}$ for each $\sigma$:

\begin{prop}\cite[Equation (2.4)]{GKS} Let $\sigma^{ac}$ and the action of $H$ be defined as before. Then
    \begin{equation}\label{eq: F0}\sum_{h \in H}[h]F_{h\sigma^{ac}}(1,\vec{z})=\prod_{v}z_{v}^{{\delta_v}-2}\prod_{v}2\text{s.e.}(([g_v]z_v^{2}-1)^{2-\delta_v}),\end{equation} where both sides of the equation are considered as formal bi-infinite series in $\mathbb{Z}[[\vec{z}^{\pm1}]][H].$
\end{prop} 
The authors of \cite{GKS} do not include a proof of this statement but we write one below, in order to show how the argument can be adapted to the two-variable case.
\begin{proof} Beginning with (\ref{eq: F0}), observe that  \[\prod_{v}z_v^{\delta_v-2}=\vec{z}^{\sigma^{ac}}.\] The two expansions of
\[([g_v]z_v^{2}-1)^{2-\delta_v}\] are 
\[\big(-\sum_{k \geq 0}[g_v^{k}]z_v^{2k}\big)^{\delta_v-2}\] and \[\big(\sum_{k\geq 0} [g_v]^{-(k+1)}z_v^{-(2k+2)}\big)^{\delta_v-2}.\] Fixing $v$, each term of 
\[z_v^{\delta_v-2}2\text{s.e.}(([g_v]z_v^{2}-1)^{2-\delta_v})\] is of the form \[a[g_v]^{n}z_v^{\delta_v-2+2n}=a[g_v]^{n}z_v^{([g_v]^{n}\sigma^{ac})_v}, a \in \mathbb{Z}.\] Therefore in the product \[\prod_{v}z_{v}^{{\delta_v}-2}\prod_{v}2\text{s.e.}(([g_v]z_v^{2}-1)^{2-\delta_v}),\]  the coefficient on each term $a\vec{z}^{\ell}$ in the group ring is $h$ where $[\ell]=[h\sigma^{ac}].$
\end{proof}

Similarly, one can show that each term of \[(z_vt^{-1})^{\delta_v-2}2\text{s.e.}(([g_v]z_v^{2}t^{-2}-1)^{2-\delta_v})\] is of the form \[a[g_v]^{n}(z_vt^{-1})^{\delta_v-2+2n}=a[g_v]^{n}(z_vt^{-1})^{([g_v]^{n}\sigma^{ac})_v}, a \in \mathbb{Z},\] which proves the following. 
\begin{prop}\label{prop: equivariant} As elements of the group ring $\mathbb{Z}[\vec{z}^{\pm 1},t^{\pm 1}][H]$, we have 
    \begin{equation}\label{eq: equivariant}\sum_{h \in H}[h]F_{h\sigma^{ac}}(\vec{z},t)=\prod_{v}(z_vt^{-1})^{\delta_v-2}\prod_{v}2\text{s.e.}([g_v](z_vt^{-1})^2-1)^{2-\delta_v}.\end{equation}
\end{prop} 

We now state and prove the formulae for $\Zhathat_0$ and $\Zhathat_{\sigma}$ of Seifert manifolds. \newpage{}
\begin{theorem}\label{thm: nstar main full}
Let $M(b;(a_{1},b_{1}),\dots,(a_{k},b_{k}))$ be a Seifert manifold with  plumbing matrix $M$. 
\begin{enumerate} 
    \item Let $\Lambda \in \mathbb{Q}$ be the prefactor as calculated in \cite[Equation (4.8)]{GKS}, and let $\mathcal{L}_{A}$ denote a Laplace transform sending terms of the form $z^{n}t^{m}\mapsto q^{\frac{n^2}{4A}}t^{m}.$ Then \[\Zhathat_0(M)(t^{2},q^{|H|})=q^{\Lambda}\mathcal{L}_{A}(\text{s.e.}(f_{0}(t,z))).\]
    \item Let $n:=A(k-2)-\sum_{i}\overline{a_i}$. Then \[\Zhathat_{\sigma}(t^2,q)=q^{\frac{\Lambda}{|H|}}\cdot \mathcal{L}_{A|H|}\text{s.e.}(f_{\sigma}(t,z)),\] where \begin{equation}\label{eq: se eqv}\sum_{h \in H}f_{h\sigma^{ac}}[h]=(z^{n}t^{2})\big(\frac{([g_1]z^{2\overline{a_i}}t^{-1}-1)\dots([g_k]z^{2\overline{a_k}}t^{-1}-1)}{([g_0]z^{2\overline{A}}t^{-1}-1)^{k-2}}\big).\end{equation}
    \end{enumerate}
\end{theorem}
\begin{proof} Starting with the first statement and applying Lemma \ref{lem: shortcut}, we have \[\Zhathat_{0}(M)(t^{2},q)=q^{\eta}\sum_{\ell \in \delta + 2\mathbb{Z}^{s}}(\prod_{i}\widehat{W}_{\delta_{i}}(\ell_{i})))q^{\frac{\ell^{t}M^{-1}\ell}{4}}t^{\ell \cdot u}\] As mentioned in Remark \ref{rem: se prod}, $(\prod_{i}\widehat{W}_{\delta_{i}}(\ell_{i})))$ identifies the coefficient on $\vec{z}^{-\ell}$ in the symmetric expansion of $\prod_{v}(z_{v}-z_{v}^{-1})^{2-\delta_{v}}$.
As in the proof of \cite[Theorem $4.1$]{GKS}   we may compute $\Zhathat_{0}(t^{2},q^{|H|})$ by replacing terms of the form \[q^{-\frac{\ell^{t}M^{-1}\ell}{4}}t^{u \cdot \ell}\] with terms of the form \[q^{-\frac{\ell^{t}\det(-M)M^{-1}\ell}{4}}t^{u \cdot \ell},\] For notational convenience, we will let \[M':=\det(-M)M^{-1}.\]

Follow the convention in \cite{GKS} that $z_{0}$ corresponds to the node and  $z_{1},\dots z_{k}$ correspond to the leaves.
Exactly as in the proof of \cite[Theorem $4.1$]{GKS}, we evaluate $\frac{\ell^{t}M'\ell}{4}$ for each $\ell$. Let $\overline{a_{ij}}:=\frac{A}{a_{i}a_{j}}$. The exponent on $q$ (multiplied by $4$ for notational convenience) becomes
 \begin{equation}\label{eq: zhathat powers} 4\ell^{t}M'\ell=\ell_{0}^{2}A+2\sum_{i\neq0}\ell_{0}\ell_{i}\overline{a_{i}}+\sum_{i\neq j}\ell_{i}\ell_{j}\overline{a_{ij}}+\sum_{i\neq0}\ell_{i}^{2}\text{adj}M_{ii}.\end{equation} The calculation above is a consequence of the relationship between $M'$ and the \textit{splice diagram} associated to the plumbing graph $\Gamma$; for more details see \cite[Theorems $3.3$ and $4.1$]{GKS}. Applying Corollary \ref{cor: cl}, we get that all terms are of the form \[C(\ell)q^{\frac{\ell^{t}M'\ell}{4}}[t^{\ell\cdot u}+t^{-\ell \cdot u}].\]

On the other hand, we calculate the symmetric expansion of $f_{0}(z,t)$. We start with
\[\text{s.e.}(f_{0}(t,z))=\left[\prod_{1\leq i \leq k}(z_{v_{i}}^{\overline{a_{i}}}t^{-1}-z_{v_{i}}^{-\overline{a_{i}}}t)\right]\text{s.e.}\left[z_{v_0}^{A}t^{-1}-(z_{v_0}^{A}t^{-1})^{-1})^{2-k}\right].\] In the resulting formal power series, all terms will be of the form
\[C(\ell)z^{A\ell_{0}+\sum_{i\geq 1}\overline{a_{i}}\ell_{i}}t^{-\ell \cdot u}\] Note that for $1 \leq i \leq k, \ell_{i}=\pm 1$. 

After the Laplace transform, the power on $q$ corresponding to each $\pm \ell$, again multiplied by $4$, becomes \begin{equation}\label{eq: laplace powers}(A\ell_{0}+\sum_{i \neq 0}\ell_{i}\overline{a_{i}})^{2}\slash A=\ell_{0}^{2}A+2\sum_{i\neq 0}\ell_{0}\ell_{i}\overline{a_{i}}+\sum_{i \neq j,i,j\geq 1}\ell_{i}\ell_{j}\overline{a_{ij}}+\sum_{i \neq 0}\ell_{i}^{2}\overline{a_{ii}},\end{equation} which we shorten to $f(\pm \ell).$ Each term of the symmetric expansion is of the form \[C(\ell)q^{\frac{f(\pm\ell)}{4}}[t^{\ell\cdot u} + t^{-\ell \cdot u}].\]

The only difference between the powers of $q$ in equations (\ref{eq: zhathat powers}) and (\ref{eq: laplace powers}) is the coefficient on $\ell_{i}^{2}$ for $i\neq 0$. This can be corrected in the prefactor by following the same process from the proof of \cite[Theorem $4.1$]{GKS}. This is because for all $\ell$ that appear as exponents in the symmetric expansion, $\ell_{i}^{2}=1$ when $1\leq i \leq k.$

As was the case in \cite{GKS}, the proof of statement (2) follows from applying the same comparison of powers using Equation (\ref{eq: equivariant}) from Proposition \ref{prop: equivariant}.

\end{proof}

\section{Quantum modularity of $\Zhathat_0$ for three exceptional fibers}\label{sec: zhathat modularity} 
This section proves Theorem \ref{thm: nstar modularity}. 
We begin with a proposition about the structure of $\mathcal{L}_{A}(\text{s.e.}(f_{0}(\omega,z)))$.
\begin{prop}\label{prop: odd periodic} Let $M=M(b;(a_{1},b_{1}),(a_{2},b_{2}),(a_{3},b_{3}))$ and define $A, \overline{a_{i}}$ as before. Let $\omega$ be a $2j$-th root of unity. Then there exists a polynomial $p(q)$ such that
\[\mathcal{L}_{A}(\text{s.e.}(f_{0}(\omega,z)))=\sum_{n\geq 0}\mathcal{C}(n;\omega)q^{\frac{n^{2}}{4A}}+p(q),\] where $\mathcal{C}(n;\omega)$ is odd and periodic in $n$. \end{prop} 

\begin{proof}
 Let $\varepsilon_{i}=\pm 1$, and let $m \in \mathbb{Z}$. Define \[c(\varepsilon_{1},\varepsilon_{2},\varepsilon_{3},m)=\begin{cases} -\frac{1}{2}\varepsilon_{1}\varepsilon_{2}\varepsilon_{3}[t^{\sum_{i}\varepsilon_{i}+m}+t^{-(\sum_{i}\varepsilon_{i}+m)}]& m \text{ odd} \\ 0 & m\text{ even.} \end{cases}\]
The function $c$ has the following symmetry: \[c(\varepsilon_{1},\varepsilon_{2},\varepsilon_{3},m)=-c(-\varepsilon_{1},-\varepsilon_{2},-\varepsilon_{3},-m).\] If $t$ is a $2j$-th root of unity $\omega$, we also have that
\[c(\varepsilon_{1},\varepsilon_{2},\varepsilon_{3},m)=c(\varepsilon_{1},\varepsilon_{2},\varepsilon_{3},m+2j).\] Let \[Y:=\{(\varepsilon_{1},\varepsilon_{2},\varepsilon_{3},m)|\varepsilon_{i}=\pm 1, m>0, m\text{ odd}\};\] \[X=\{(\varepsilon_{1},\varepsilon_{2},\varepsilon_{3},m)|\varepsilon_{i}=\pm 1, m\in \mathbb{Z},m\text{ odd, }\sum_{i}\varepsilon_{i}\overline{a_{i}}+mA>0\}.\] Observe that both $Y-X$ and $X-Y$ are finite sets (which may be empty). Given $(\varepsilon_{1},\varepsilon_{2},\varepsilon_{3},m)$, let $n:=\sum_{i}\varepsilon_{i}\overline{a_{i}}+mA.$
We may now rewrite 
    \begin{equation}\label{eq: lac}\mathcal{L}_{A}(\text{s.e.}(f_{0}(t,z)))=\sum_{Y}c(\varepsilon_{1},\varepsilon_{2},\varepsilon_{3},m)q^{\frac{n^{2}}{4A}}.\end{equation} 
    On the other hand, we may re-index and rewrite \[\sum_{X}c(\varepsilon_{1},\varepsilon_{2},\varepsilon_{3},m)q^{\frac{n^{2}}{4A}}=\sum_{n\geq 0}\mathcal{C} (n;\omega)q^{\frac{n^{2}}{4A}},\] where $\mathcal{C}(n;\omega)q^{\frac{n^{2}}{4A}}$ is $2Aj$-periodic and odd in $n$. Putting everything together,
    \[\sum_{y \in Y}c(y)q^{\frac{n^{2}}{4A}}=\sum_{x\in X}c(x)q^{\frac{n^{2}}{4A}} + \sum_{ z \in Y-X}c(z)q^{\frac{n^{2}}{4A}}-\sum_{w \in X-Y}c(w)q^{\frac{n^{2}}{4A}}.\] 
    As both $X-Y$ and $Y-X$ are finite sets, we may write \[\mathcal{L}_{A}(\text{s.e.}(f_{0}(t,z)))=\sum_{n\geq 0}\mathcal{C}(n;\omega)q^{\frac{n^{2}}{4A}}+p(q),\] where $p(q)$ is a polynomial in $q$. 
\end{proof}

\begin{remark} From here, one completes the proof of Theorem \ref{thm: nstar modularity} by applying the same argument used in \cite[Section $6.1$]{LM} to show that \[\sum_{n\geq 0}\mathcal{C}(n;\omega)q^{\frac{n^2}{4A}}\] is a quantum modular form of weight $\frac{1}{2}$ with respect to the subgroup $\Gamma(4a_1a_2a_3j^{2})$. 
\end{remark}

   \begin{remark}\label{rem: radial3}
       Because $\mathcal{C}(n;\omega)$ is periodic with mean value zero, one may take radial limits toward roots of unity as in \cite[Section $4.1$]{LM}. Set $\zeta=\omega^{2}$, a $j$th root of unity. For any root of unity $\xi$, $\lim_{t \searrow 0}\Zhathat(\zeta,\xi e^{-t})$ exists and can be expressed as a finite sum in terms of $\mathcal{C}(n;\omega)$. These limits can be thought of as $\zeta$-deformations of Witten-Reshetikhin-Turaev invariants, extending the result of \cite[Theorem $1.1$]{LM}.
       \end{remark}

\subsection{Radial limits as $q \to 1$}

Toward a topological understanding of the newly $\zeta$-deformed Witten-Reshetikhin-Turaev invariants, we first compute limits as $q \nearrow 1$. We accomplish this by studying $\Zhathat(t^2,q)$ as a series in $t$ with polynomial coefficients in $q$, which is possible due to the shape of the plumbing graph. We do not expect this argument to extend beyond $3$ exceptional fibers.

By investigating $\Zhathat$ as a series in $t$, we can compute the limit as $q \nearrow 1$ for all $t \in \mathbb{C}$, rather than just roots of unity. We find that in this case, the limit depends only on the choice of $t$ and not on the manifold itself.

\begin{prop}\label{prop: q1}
    Let $t \in \mathbb{C}$ and let $M$ be a Seifert manifold with $3$ exceptional fibers. Then 
    \[\lim_{q\nearrow 1}\Zhathat_0(t,q)=\lim_{x\searrow 0}\Zhathat_0(t,e^{-x})=t+t^{-1}-2.\]
\end{prop}

Toward a proof of this proposition, we now rearrange the $\Zhathat$ series for Seifert manifolds with $3$ exceptional fibers as a series in $t$. Recall from equation \ref{eq: lac} that \[\Zhathat_0(t^2,q)=q^{\Lambda}\sum_{Y}c(\varepsilon_1,\varepsilon_2,\varepsilon_3,m)q^{\frac{n^2}{4A}}.\] Fixing $j \in \mathbb{Z}$, the definition of $c(\varepsilon_1,\varepsilon_2,\varepsilon_3,m)$ guarantees 
there are only finitely many tuples $(\varepsilon_1,\varepsilon_2,\varepsilon_3,m)\in Y$ for which $c(\varepsilon_1,\varepsilon_2,\varepsilon_3,m)$ contains a power $t^{j}$. This is because $m > 0$ and $-3 \leq \sum_{i}\varepsilon_i \leq 3$. Therefore we can rewrite \[\Zhathat_0(t^2,q)=q^{\Lambda}\sum_{Y}c(\varepsilon_1,\varepsilon_2,\varepsilon_3,m)=\sum_{j \in \mathbb{Z}}c_j(q)t^{j},\] where $c_j(q)$ is a polynomial in $q$. By definition, $c_j(q)=c_{-j}(q)$. We now have \[\lim_{q\nearrow 1}\Zhathat_0(t^2,q)=\sum_{j \in \mathbb{Z}}\lim_{q\nearrow 0}c_{j}(q)t^{j}.\]

\begin{lemma}\label{lem: c0} Focusing on coefficients on $t^0$ and $t^{\pm 2}$, we have
\begin{align*}
    \lim_{q\nearrow 1}c_0(q)&=2,\\
    \lim_{q \nearrow 1}c_{-2}(q)=\lim_{q\nearrow 1}c_2(q)&=1.
\end{align*}
\end{lemma}
\begin{proof} Starting with $c_0$, the only elements of $Y$ contributing multiples of $t^0$ to the sum are 
\[(-1,-1,-1,3),(1,-1,-1,1),(-1,1,-1,1),\text{ and } (-1,-1,1,1).\] The coefficient contributed by $(-1,-1,-1,3)$ is 
\[\frac{1}{2}[q^*t^0 + q^*t^0],\] where $*$ is a rational number given by $\frac{(\sum_{i}\varepsilon_{i}\overline{a_i}+mA)^2}{4A}$. As $q \nearrow 1$, this element contributes $1$ to the limit. Repeating the process for each of the other three relevant elements of $Y$, we see that they each contribute $-1$ to the limit, therefore, in total, the four elements give a limit of $-2$.
A similar argument shows that \[\lim_{q \nearrow 1}c_{-2}(q)=\lim_{q\nearrow 1}c_2(q)=1.\]
\end{proof}
We now use Lemma \ref{lem: c0} to prove Proposition \ref{prop: q1}. 
\begin{proof}
    After applying Lemma \ref{lem: c0}, it remains to show that for $|j|\geq 4$, we have $\lim_{q \nearrow 1}c_j(q)=0$. This follows from a straightforward inspection of the coefficients contributed by the eight relevant elements of $Y$, which are:
\begin{align*}
    (-1,-1,-1,j+3),\\ (-1,-1,1,j+1),(-1,1,-1,j+1),(1,-1,-1,j+1)\\ (-1,1,1,j-1),(1,-1,1,j-1),(1,1,-1,j-1),\\\text{ and }(1,1,1,j-3).
\end{align*} One can show that as $q \nearrow 1$ the sum of these eight coefficients goes to zero.
Putting everything together, we now have that \[\lim_{q \nearrow 1}\Zhathat_0(t^2,q)=t^2+t^{-2} -2,\] implying \[\lim_{q \nearrow 1}\Zhathat_0(t,q)=t+t^{-1}-2.\]\end{proof} 

\section{Calculation of $(t,q)$-series invariants}\label{sec: series calc} This section presents calculations of the $(t,q)$-series invariants $P_{W,0}^{\infty}(t,q)$, in terms of $W$, for three infinite families of manifolds. 
\subsection{Seifert manifolds with three exceptional fibers} For this collection of manifolds, the invariants $P_{W,0}^{\infty}$ can be thought of as a one-parameter family coming from \textit{asymmetric expansions}, in which the $\Zhathat$ invariant is realized from the unique symmetric expansion. 

\begin{definition}\label{def: ae}
    Given a function $f(z)\colon \mathbb{C} \to \mathbb{C}$ with Laurent expansions of $\sum_{k}a_{k}z^{k}$ centered at $z=0$ and $\sum_{k}b_{k}z^{-k}$ centered at $z=\infty$, fix $y \in \mathbb{R}$ and define the \textit{asymmetric expansion}, denoted a.e., to be the formal power series given by \[\text{a.e.}_{y}(f):=(1-y)\sum_{k}a_{k}z^{k}+y\sum_{k}b_{k}z^{-k}.\] When $y=\frac{1}{2},$ this becomes the \textit{symmetric expansion} used in \cite{GKS}.
\end{definition} 

\begin{theorem}\label{thm: 3star series calc full}
    Given a Seifert manifold $M(b;(a_1,b_1),(a_2,b_2),(a_3,b_3))$, fix an admissible family and define $\Lambda,A,f_{0}(t,z),$ and $\mathcal{L}_{A}$ as in Theorem \ref{thm: nstar main}. Then \begin{equation}P^{\infty}_{W,0}(M)(t^{2},q^{|H|})=q^{\Lambda}(\mathcal{L}_{A}(\text{a.e.}_{W_{3}(1)}(f_0(t,z))).\end{equation}\label{theorem eq}
\end{theorem}

\begin{proof}
We begin with the asymmetric expansion of $(z^{A}t^{-1}-(z^{A}t^{-1})^{-1})$, which is the denominator of $f_{0}(t,z)$. For $|z^{A}t^{-1}|<1$ we have \[-\sum_{i \geq 0} (z^{A}t^{-1})^{2i+1}\] and for $|z^{A}t^{-1}|>1$ we have \[\sum_{i \geq 0}(z^{A}t^{-1})^{-(2i+1)}.\] Expanding on $|z^{A}t^{-1}|<1$ we get 
\[-\prod_{j=1}^{3}(z^{\overline{a_{i}}}t^{-1}-z^{-\overline{a_{i}}}t)\sum_{i\geq 0}(z^{A}t^{-1})^{2i+1},\] which we can rewrite as 

\[-\sum_{\varepsilon_{j}=\pm 1, i \geq 0}(\varepsilon_{1}\varepsilon_{2}\varepsilon_{3})z^{\sum_{j=1}^{k}\varepsilon_{j}\overline{a}_{j}+(2i+1)A}t^{-\sum_{j=1}^{k}\varepsilon_{j}-(2i+1)}.\] On $|z^{A}t^{-1}|>1$, we have
\[\prod_{j=1}^{3}(z^{\overline{a_{i}}}t^{-1}-z^{-\overline{a_{i}}}t))\sum_{i\geq 0}(z^{A}t^{-1})^{-(2i+1)},\] which is equal to \[\sum_{\varepsilon_{j}=\pm 1, i \geq 0}(\varepsilon_{1}\varepsilon_{2}\varepsilon_{3})z^{\sum_{j=1}^{k}\varepsilon_{j}\overline{a}_{j}-(2i+1)A}t^{\sum_{j=1}^{k}\varepsilon_{j}+(2i+1)}.\] 

Let $\ell:=(\varepsilon_1,\varepsilon_2,\varepsilon_3,2i+1)$, and define \[f(\ell):=\sum_{j}\varepsilon_{j}\overline{a}_{j}+(2i+1)A.\]

 Applying the Laplace transform, the powers of $q$ corresponding to $\ell$ in the expansion on $|z^{A}t^{-1}|<1$ are equal to those corresponding to $-\ell$ in the expansion on $|z^{A}(t)|>1$, so we may combine the two to write \begin{equation}\label{eq: ae formulation} \mathcal{L}_{A}(\text{a.e.}_{W_{3}(1)}f_0(t,q))=\sum_{\varepsilon_j,i\geq 0}-(\prod_{i}\varepsilon_i)q^{\frac{f(\ell)}{4A}}\left[W_{3}(1)t^{\ell \cdot u} + (1-W_{3}(1)t^{-\ell \cdot u} \right].\end{equation}
On the other hand, working with the formula from Lemma \ref{lem: shortcut}, we may write \[P_{W,0}^{\infty}(M)(t^{2},q)=q^\eta\sum_{\ell \in \delta + 2\mathbb{Z}^{s}}(\prod_{i}W_{\delta_{i}}(\ell_{i})))q^{\frac{\ell^{t}M^{-1}\ell}{4}}t^{\sum_{i}\ell_{i}}.\]

Because $\delta=(3,1,1,1,2,2,\dots),$ the axioms of admissible families imply that the only $\ell$ for which $\prod_{i}W_{\delta_{i}}(\ell_{i})$ is nonzero are those for which $\ell=(m, \pm 1, \pm 1, \pm 1, 0,\dots)$, and $W_{3}(m)\neq 0$. In this case, 
\[\prod_{i}W_{\delta_{i}}(\ell_i)=W_{3}(m)(-\varepsilon_{1})(-\varepsilon_{2})(-\varepsilon_{3}).\] Note that since $\delta_{1}=3$, $m$ must be odd. By the axioms of admissible families,\[W_{3}(m)=\begin{cases}W_{3}(1) & m\text{ odd}, m \geq 1 \\ W_{3}(1)-1 & m\text{ odd}, m \leq -1 \\ 0 &\text{otherwise.}\end{cases}\]

To calculate $P_{W,[k]}^{\infty}(t^{2},q^{|H|})$ we make the same replacement of powers of $q$ as in the proof of Theorem \ref{thm: nstar main full}:
\begin{equation}\label{eq: general pfk}P_{W,0}^{\infty}(t^{2},q^{|H|})=q^{-|H|\eta}\sum_{m \geq 0, \varepsilon_{i}=\pm 1}(\prod_{i}\varepsilon_{i})q^{-\frac{\ell^{t}M'\ell}{4}}\big[-W_{3}(1)t^{\ell \cdot u}+ (W_{3}(1)-1)t^{-\ell \cdot u}\big].\end{equation}

The only difference between Equations \ref{eq: ae formulation} and \ref{eq: general pfk} is the powers of $q$ and the prefactor, which we correct using the same trick as in the proof of Theorem \ref{thm: nstar main full}.
\end{proof}

\begin{cor}\label{cor: 3star}
    For a Seifert manifold $M(b;(a_1,b_1),(a_2,b_2),(a_3,b_3))$ and any admissible family $W$, $P_{W,0}^{\infty}(1,q)=\Zhat_0(q)$. 
\end{cor}

\begin{remark}
    Theorem \ref{thm: 3star series calc full} suggests a possible interpretation of the role of $W$ in $P_{W,[k]}^{\infty}(t,q)$. In the case of Seifert manifolds with three exceptional fibers, the $\Zhathat$ series is the unique series resulting from a symmetric expansion. The following calculations further motivate an understanding of the $P_{W,0}^{\infty}$ as an asymmetric deformation of the $\Zhathat$ series. 
\end{remark}

\subsection{H-shaped graphs}
\begin{example}\label{ex: Hshape}
Let $M$ be a manifold with an $H$-shaped graph consisting of two degree-$3$ nodes and four leaves. Starting with the formula given by 
Lemma \ref{lem: shortcut} and further restricting $\ell_{i}=\pm 1$ for $2 \leq i \leq 5$, we have 

\[P_{W,0}^{\infty}(M)(t^{2},q)=q^\eta\sum_{\ell_{i}=(m,n,\varepsilon_{1},\varepsilon_{2},\varepsilon_{3},\varepsilon_{4},0,\dots,0)}W_{3}(m)W_{3}(n)(\prod_{i}-\varepsilon_{i})q^{\frac{\ell^{t}M^{-1}\ell}{4}}t^{\ell\cdot u}, \] where $m,n$ are odd integers. 

Because $W_{3}(m)$ and $W_{3}(n)$ depend only on the sign of $m$ and $n$, we can split the sum into four cases. These four sums coincide with the four possible regions of $\mathbb{C}^{s}$ on which one can expand \[\prod_{i}(z_{i}-z_{i}^{-1})^{2-\delta_{v}}\] for an $H$-shaped graph, where $z_0$ and $z_1$ are the two nodes:

\begin{center} 
\begin{tabular}{l|l|l}
Region of Laurent expansion & Signs of $(m,n)$ & Coefficient on corresponding sum \\ \hline 
$|z_{0}|<1,|z_{1}|<1$ & $(+,+)$ & $(W_{3}(1)-1)^{2}$\\
$|z_{0}|<1, |z_{1}|>1$ & $(+,-)$ &$(W_{3}(1))(W_{3}(1)-1)$\\
$|z_{0}|>1, |z_{1}|<1$ & $(-,+) $&$(W_{3}(1))(W_{3}(1)-1)$\\
$|z_{0}|>1,|z_{1}|>1$& $(-,-)$  &$(W_{3}(1))^{2}.$
\end{tabular}
\end{center} Let \[Y_+:=\{(m, n, \varepsilon_{1},\varepsilon_2,\varepsilon_3,\varepsilon_4)|m,n>0,m,n\text{ odd, }, \varepsilon_{i}=\pm 1\}\] and 
\[Y_-:=\{(m, n, \varepsilon_{1},\varepsilon_2,\varepsilon_3,\varepsilon_4)|m>0, n < 0,m,n\text{ odd, }, \varepsilon_{i}=\pm 1\}\]
From here we may use symmetries associated with $\ell \mapsto -\ell$ to arrive at the following:

\begin{align*} P_{W,0}^{\infty}(M)(t^{2},q)&=q^\eta\sum_{Y_{+}}(\prod_{i}\varepsilon_i)q^{\frac{\ell^{t}M^{-1}\ell}{4}}[(W_{3}(1))^{2}t^{\ell \cdot u } +(W_{3}(1)-1)^{2}t^{-\ell \cdot u}] \\ & +q^\eta(W_{3}(1))(W_{3}(1)-1)\sum_{Y_-}(\prod_{i}\varepsilon_i)q^{\frac{\ell^{t}M^{-1}\ell}{4}}[t^{\ell\cdot u}+t^{-\ell\cdot u}].\end{align*} 

The coefficients on all four sums are equal if and only if $W_{3}(1)=\frac{1}{2}$, in which case the $\Zhathat$ invariant is recovered and the coefficients are all equal to $\frac{1}{4}$. In this sense the $P_{W,[k]}^{\infty}$ invariant associated to this family of manifolds can also be interpreted as an asymmetric deformation of the $\Zhathat$ invariant.
\end{example} 

\subsection{$4$-leg star graphs}
\begin{example}\label{ex: 4star}
Let $M$ be manifold whose plumbing graph has one $4$-valent node, $4$ leaves, and any amount of joints. Fix an admissible family of functions $W$. It was observed by Peter Johnson, who gave a general formula for $W_{n}(r)$ in terms of $\{W_{n}(0),W_{n}(1)\}_{n \in \mathbb{N}}$, that \[W_{4}(m)=\begin{cases} W_{4}(0)+(W_{3}(1)-1)\frac{n}{2} & m\text{ even, }m\leq 0 \\W_{4}(0)+W_{3}(1)\frac{n}{2} & m\text{ even, }m>0 \\ 0 & \text{otherwise.} \end{cases}\] When $W=\widehat{W}$, we have \[\widehat{W}_{4}(n)=\begin{cases} \frac{|n|}{4} & n\text{ even} \\ 0 & \text{otherwise.}\end{cases}\] Using Lemma \ref{lem: shortcut} and further restricting to $\ell_{i}=\pm 1$ for $1 \leq i \leq 4$ we may write:
\begin{align*}P_{W,0}^{\infty}(M)(t^{2},q)=&q^\eta\sum_{\varepsilon_{i}=\pm 1,n\text{ even, }n < 0 }(\prod_{i}\varepsilon_{i})q^{\frac{-\ell^{t}M'\ell}{4}}(W_{4}(0)+(W_{3}(1)-1)\frac{n}{2})t^{\ell\cdot u}\\
&+q^\eta\sum_{\varepsilon_{i}=\pm 1,n\text{ even, }n > 0}(\prod_{i}\varepsilon_{i})q^{-\frac{\ell^{t}M'\ell}{4}} (W_{4}(0)+W_{3}(1)\frac{n}{2})t^{\ell\cdot u}\\
&+q^\eta W_{4}(0)\sum_{\varepsilon_{i}=\pm 1}(\prod_{i}\varepsilon_{i})q^{-\frac{\ell^{t}M'\ell}{4}}t^{\sum_{i}\varepsilon_{i}}.\end{align*}
Regrouping terms, we have
\begin{align*} P_{W,0}^{\infty}(M)(t^{2},q)&= (1-W_{3}(1))q^\eta\sum_{\varepsilon_{i}=\pm 1,n\text{ even, }n < 0}(\prod_{i}\varepsilon_{i})q^{-\frac{\ell^{t}M'\ell}{4}}\frac{|n|}{2}t^{\ell\cdot u}\\
& +W_{3}(1)q^\eta\sum_{\varepsilon_{i}=\pm 1,n\text{ even, }n > 0}(\prod_{i}\varepsilon_{i})q^{-\frac{\ell^{t}M'\ell}{4}}\frac{|n|}{2}t^{\ell\cdot u}\\
& + W_{4}(0)q^\eta\sum_{\varepsilon_{i}=\pm 1,n\text{ even}}(\prod_{i}\varepsilon_{i})q^{\frac{\ell^{t}M'\ell}{4}}t^{\ell\cdot u}.\end{align*} 
\end{example}

After re-indexing, the first sum above can be seen as corresponding to the expansion of \begin{equation}\label{eq: expansion} \prod_{v}(z_{v}-z_{v}^{-1})^{2-\delta_{v}}\end{equation} on $|z_0|<1$ and the second sum can be seen as corresponding to the expansion on (\ref{eq: expansion}) on $|z_{0}|>1$. In this sense Example \ref{ex: 4star} can also be seen as an asymmetric deformation of the $\Zhathat$ invariant. If $W=\widehat{W}$, then $W_{3}(1)=\frac{1}{2}$ and $W_{4}(0)=0$, so the symmetric expansion is recovered.

\subsection{New $q$-series invariants}

Recall from Corollary \ref{cor: 3star} that when $M$ is a Seifert manifold with three exceptional fibers, $P_{W,0}^{\infty}(1,q)=\Zhat_{0}(q)$, regardless of $W$. The following two lemmas clarify that this does not hold in general:

\begin{lemma}
    Let $M$ be a negative definite plumbed manifold with an $H$-shaped plumbing graph. Let $y_{1}:=2W_{3}(1)^{2}-2W_{3}(1)+1$, $y_{2}:=y_{1}-1$, and $\varepsilon:=\prod_{i}\varepsilon_{i}.$ Then \[P_{W,0}^{\infty}(M)(1,q)=q^\eta\left[y_{1}\sum_{\varepsilon_{i}=\pm 1,m,n>0}\varepsilon q^{\frac{\ell^{t}M^{-1}\ell}{4}}+y_{2}\sum_{\varepsilon_{i}=\pm 1, m>0,n<0}\varepsilon q^{\frac{\ell^{t}M^{-1}\ell}{4}}\right].\]
\end{lemma}  

\begin{lemma}
    Let $M$ be a manifold whose plumbing graph is star-shaped with $4$ legs. Then \[P_{W,0}^{\infty}(1,q)=\Zhat_{0}(q)+W_{4}(0)q^\eta\sum_{\varepsilon_{i}=\pm 1,n\text{ even}}\varepsilon_{1}\varepsilon_{2}\varepsilon_{3}\varepsilon_{4}q^{-\frac{\ell^{t}M'\ell}{4}}.\]
\end{lemma}

\section{Modularity properties of the $(t,q)$-series for three exceptional fibers}\label{sec: modularity} 
 This section proves Theorem \ref{thm: 3star modularity}, establishing $\mathcal{L}_{A}(\text{a.e.}_{W_{3}(1)}(f_{0}(\omega,z)))$ as a sum of quantum modular and modular forms when $M=M(b;(a_1,b_1),(a_2,b_2),(a_3,b_3))$ and $\omega$ is a $2j$-th root of unity. To accomplish this we decompose the series into a linear sum of $q$-series, one whose coefficients are given by a function  $\mathcal{C}(n;\omega)$ which is $2Aj$-periodic and odd in $n$, and the other whose coefficients are given by $\mathcal{D}(n;\omega)$, which is $2Aj$-periodic and even in $n$. The quantum modularity of the sum involving $\mathcal{C}(n;\omega)$ follows from Section \ref{sec: zhathat modularity}. The modularity of the sum involving $\mathcal{D}(n;\omega)$ follows from a careful application of \cite[Lemma $2.1$]{goswamiosburn}. We now show how to obtain $\mathcal{D}(n;t)$.

\begin{definition}Consider a function $f$ with Laurent expansions of $\sum_{k}a_{k}z^{k}$ centered at $z=0$ and $\sum_{k}b_{k}z^{-k}$ centered at $z=\infty$. Let the \textit{series difference}, denoted $\text{s.d.}$, be the formal power series given by 
\[\text{s.d.}(f)=\sum_{k}a_{k}z^{k} - \sum_{k}b_{k}z^{-k}.\]
\end{definition} Observe that \[\text{a.e.}_{y}(f)= \text{s.e.}(f)+(\frac{1}{2}-y)\text{s.d.}(f)\] so we may write 

\[\mathcal{L}_{A}(\text{a.e.}_{W_{3}(1)}(f_{0}(\omega,z))))=\mathcal{L}_{A}(\text{s.e.}(f_0)(\omega,z)) + (\frac{1}{2}-W_{3}(1))\mathcal{L}_{A}(\text{s.d.}(f_{0}(\omega,z)).\] From Section \ref{sec: zhathat modularity} we know that \[\mathcal{L}_{A}(\text{s.e.}(f_{0}(t,z)))=p_{1}(q) + \sum_{n \geq 0}\mathcal{C}(n;t)q^{\frac{n^{2}}{4A}},\] where $\mathcal{C}(n;t)$ is odd and periodic in $n$ and $p_{1}(q)$ is a polynomial. In the language of Theorem \ref{thm: 3star modularity}, $\mathcal{L}_{A}(\text{s.e.}(f_{0}(t,z)))$ is $\theta_{f}$.

{\begin{prop}\label{prop: even periodic}
    For $M=M(b;(a_1,b_1),(a_2,b_2),(a_3,b_3))$ and $\omega$ a $2j$-th root of unity, \[\mathcal{L}_{A}(s.d.(f_{0}(\omega,z)))=p_{2}(q) + \sum_{n\geq 0}\mathcal{D}(n;\omega)q^{\frac{n^{2}}{4A}},\] where $\mathcal{D}(n;\omega)$ is even and $2Aj$-periodic in $n$ and $p_{2}(q)$ is a polynomial. 
\end{prop} 
\begin{remark} In the language of Theorem \ref{thm: 3star modularity}, \[\sum_{n\geq 0}\mathcal{D}(n;\omega)q^{\frac{n^{2}}{4A}}=\Theta_{f},\] and $p_{1}(q)+p_{2}(q)=p(q).$\end{remark}

\begin{proof} 
The expansions of $f_{0}(t,z)$ on $|z^{A}t^{-1}|<1$, and $|z^{A}t^{-1}|>1$, respectively, are \[-\sum_{\varepsilon_{i}=\pm1,i\geq 0}\varepsilon_{1}\varepsilon_{2}\varepsilon_{3}z^{(\sum_{j}\varepsilon_{j}\overline{a_{j}}+(2i+1)A)}t^{-\sum_{j}\varepsilon_{j}-(2i+1)}\] and \[\sum_{\varepsilon_{i}=\pm 1,i\geq 0}\varepsilon_{1}\varepsilon_{2}\varepsilon_{3}z^{\sum_{j}\varepsilon_{j}\overline{a_{j}}-(2i+1)A}t^{-\sum_{j}\varepsilon_{j} +(2i+1)}.\] Therefore we may write \[\mathcal{L}_{A}(\text{s.d.}(f_{0}(t,z))=\sum_{\varepsilon_{i}=\pm 1, m > 0, m\text{ odd}}d(\varepsilon_{1},\varepsilon_{2},\varepsilon_{3},m)q^{\frac{n^{2}}{4A}},\] where $n:=\sum_{i}\varepsilon_{i}\overline{a_{i}}+mA$ and \[d(\varepsilon_{1},\varepsilon_{2},\varepsilon_{3},m)=\begin{cases} \varepsilon_{1}\varepsilon_{2}\varepsilon_{3}[t^{\sum_{i}\varepsilon_{i}+m}-t^{-(\sum_{i}\varepsilon_{i}+m)}] & m\text{ odd} \\ 0 & m\text{ even.} \end{cases}\]

Observe that $d(-\varepsilon_{1},-\varepsilon_{2},-\varepsilon_{3},-m)=d(\varepsilon_{1},\varepsilon_{2},\varepsilon_{3},m),$ and for $\omega$ a $2j$-th root of unity, $d(\varepsilon_1,\varepsilon_2,\varepsilon_3,m)=d(\varepsilon_1,\varepsilon_2,\varepsilon_3,m+2j).$ 

Let $X$ and $Y$ refer to the same sets as in the proof of Proposition \ref{prop: odd periodic}. Observe  
\begin{equation}\label{eq: d sum} \sum_{X}d(\varepsilon_{1},\varepsilon_{2},\varepsilon_{3},m)q^{\frac{n^{2}}{4A}}=\sum_{n\geq 0}\mathcal{D}(n;\omega)q^{\frac{n^{2}}{4a}},\end{equation} where $\mathcal{D}(n;\omega)$ is $2Aj$ periodic and even in $n$. Then we may write 
\[\mathcal{L}_{A}(\text{s.d.}(f_{0}(\omega,z))=\sum_{n \geq 0}\mathcal{D}(n;\omega)q^{\frac{n^{2}}{4A}}+\sum_{z \in Y-X}d(z)q^{\frac{n^{2}}{4A}}-\sum_{w \in X-Y}d(w)q^{\frac{n^{2}}{4A}},\] and since $X-Y$ and $Y-X$ are finite sets. The two corresponding sums form a polynomial $p_{2}(q)$, and we arrive at the desired decomposition.
\end{proof}

Below we describe how \cite[Lemma $2.1$]{goswamiosburn} and \cite[Lemma $5.2$]{LM} are applied to (\ref{eq: d sum}) to arrive at the modularity result.  

\begin{lemma}
    Let $n$ be such that $\mathcal{D}(n;t)\neq 0.$ Then \[n^{2}\equiv k\bmod{4A}.\]
\end{lemma} 
\begin{proof} Let $n_1,n_2$ be such that $\mathcal{D}(n_i;t)\neq 0$. Then there exists some $(\varepsilon_{1},\varepsilon_{2},\varepsilon_{3},m_1),$ $m_1$ odd, such that $n_1=\sum_{i}\varepsilon_{i}\overline{a_{i}}+m_1A$, and some $(\varepsilon'_{1},\varepsilon'_{2},\varepsilon'_{3},m_2),$ $m_2$ odd, such that $n_2=\sum_{i}\varepsilon'_{i}\overline{a_{i}}+m_2A$. A straightforward calculation verifies that \[n_{1}^{2}-n_{2}^{2}\equiv 0\bmod{4A.}\]
\end{proof}
\begin{lemma}
    Let $k_{i}:=4A(i)+k$, for $0 \leq i \leq (j-1).$ If $n$ is such that $\mathcal{D}(n;t)\neq 0$, then \[n^{2}\equiv k_{i}\bmod{4Aj}\] for some $0 \leq i \leq (j-1).$
\end{lemma}

\begin{definition}[Section $1$ of \cite{goswamiosburn}] Let $1 \leq k_{i} \leq 4Aj. $ Define \[S_{i}:=\{1 \leq k \leq 2Aj|k^{2} \equiv k_{i}\bmod{4Aj}\}\] and \[\mathcal{S}_{i}:=S_{i} \cup \{4Aj- k | k \in S_{i}\}.\]
\end{definition}

\begin{lemma} For a fixed manifold $M$, $\mathcal{D}(n;\omega)$ is supported on $\bigcup_{i}\mathcal{S}_{i}.$
\end{lemma}

\begin{definition} Let $\mathcal{D}_{i}(n;\omega)$ be the restriction of $\mathcal{D}$ to $\mathcal{S}_{i}.$
 \end{definition}

 Note that \[\sum_{n\geq 0}\mathcal{D}(n;\omega)q^{\frac{n^{2}}{4A}}=\sum_{1 \leq i \leq (j-1)}\sum_{n \geq 0}\mathcal{D}_{i}(n;\omega)q^{\frac{n^{2}}{4A}}.\] Decomposing and $\mathcal{D}$ into $j$ summands will allow us to realize each associated $q$-series as a modular form. Then we will show that these modular forms transform together with respect to $\Gamma(4Aj).$
The following result of \cite{goswamiosburn}, rewritten in terms of our notation, may be applied to each $\mathcal{D}_{i}(n;\omega)$:

\begin{lemma}[\cite{goswamiosburn}, Lemma $2.1$] Let $\mathcal{D}_{i}(n;\omega)$ be an even, $2Aj$-periodic function supported on $\mathcal{S}_{i}.$ Then \begin{equation}\label{eq: modular} \theta_{\mathcal{D}}:= \sum_{n \geq 0}\mathcal{D}(n;\omega)q^{\frac{n^{2}}{4Aj}}\end{equation} is a modular form of weight $\frac{1}{2}$ with respect to $\Gamma_{1}(4Aj)$ and \[\theta_{\mathcal{D}}(\gamma \tau) = e^{\frac{\pi i a b k_{i}}{2Aj}}(\frac{4cAj}{d})\varepsilon^{-1}_{d}(cz+d)^{\frac{1}{2}}\theta_{\mathcal{D}}(\tau).\]
\end{lemma}

Here the multiplier depends on $k_{i}$. However, if we further restrict to the subgroup $\Gamma(4Aj)$, $b\equiv 0 \bmod{4Aj}$ and this dependence is eliminated. Therefore, the $j$ summands transform as one under the action of $\Gamma(4Aj) \leq \Gamma_{1}(4Aj).$ Finally, we apply the following lemma to conclude Theorem \ref{thm: nstar modularity}:

\begin{lemma} From Definition \ref{def: mod} it follows that if \[f(q):=\sum_{n \geq 0}\mathcal{D}(n)q^{\frac{n^{2}}{4Aj}}\] is a modular (resp. quantum) modular form of weight $\frac{1}{2}$ with respect to $\Gamma(4Aj)$, then \[f(q^{j})=\sum_{n\geq 0}\mathcal{D}(n)q^{\frac{n^{2}}{4A}}\] is a modular (resp. quantum) modular form of weight $\frac{1}{2}$ with respect to $\Gamma(4Aj^{2}).$
\end{lemma}

\section{Example: The $(t,q)$-series for $M(1;(2,1),(3,1),(8,1))$.}\label{sec: ex calc}
This manifold's one-variable series $\Zhat_{\sigma}$ and $\Zhat_0$ are studied in \cite[Example 6.5]{gukovmodularity}. 

\subsection{The $\Zhathat_{\sigma}(t,q)$ series}

The plumbing matrix $M$ and adjugate matrix $M'$ are:

\[M=\begin{bmatrix} -1 & 1 & 1 & 1 \\ 1 & -2 & 0 & 0 \\ 1 & 0 & -3 & 0 \\ 1 & 0 & 0 & -8 \end{bmatrix}; \qquad M'=\begin{bmatrix} -48 & -24 & -16 & -6\\ -24 & -13 & -8 & -3 \\ -16 & -8 & -6 & -2 \\ -6 & -3 & -2 & -1 \end{bmatrix}\]

One can check that $H=H_1(M;\mathbb{Z})\cong \mathbb{Z}^4\slash M\mathbb{Z}^4 \cong \mathbb{Z}\slash2\mathbb{Z}$. As elements of $H$, $[g_0]=[(1,0,0,0)]$ and $[g_2]=[(0,0,1,0)]$ are trivial. Furthermore $[g_1]=[g_3]$, and $|[g_1]|=2$.

We have \begin{align*}
    A&=48\\
    \overline{a_1}&=24\\
    \overline{a_2}&=16\\
    \overline{a_3}&=6\\
    n:=A(k-2)-\sum_{i}\overline{a_i}&=48-24-16-6=2
\end{align*}

Applying the formula (\ref{eq: se eqv}) from Theorem \ref{thm: nstar main full} yields
\[\sum_{h \in H} f_{h\sigma^{ac}}[h]=(z^nt^2)\frac{([g_1]z^{48}t^{-2}-1)(z^{32}t^{-2}-1)([g_1]z^{12}t^{-2}-1)}{(z^{96}t^{-2}-1)}\]

Expanding the numerator and combining like terms, we have 
\[f_{\sigma^{ac}}=(z^nt^2)\frac{z^{92}t^{-6}-z^{60}t^{-4}-1}{z^{96}t^{-2}-1}\]

and \[f_{g_1\cdot \sigma^{ac}}=-(z^nt^2)\frac{z^{80}t^{-4}-z^{44}t^{-4}}{z^{96}t^{-2}-1}.\]

\subsection{The $P_{W,0}(t,q)$ series}
By Theorem \ref{thm: 3star main} \[P_{W,[0]}^{\infty}(M)(t^{2},q^2)=\Zhathat_{0}(t^{2},q^{2})(M)+q^{\Lambda}[p(q)+\sum_{n\geq 0}\mathcal{D}(n;t)q^{\frac{n^{2}}{192}}].\]

The congruence classes of $n\bmod{96}$ on which $\mathcal{C}(n;t)$ and $\mathcal{D}(n;t)$ are supported are given by the eight possible triples of $\varepsilon$:
\\
\begin{center} 
\begin{tabular}{l|l|c}
$(\varepsilon_{1},\varepsilon_{2},\varepsilon_{3})$ & $n:=\sum_{i}\varepsilon_{i}\overline{a_{i}}+48$ & $n\bmod{96}$ \\ \hline 
$(1,1,1)$ & $94$ & $-2$ \\
$(1,1,-1)$ & $82$ & $-14$ \\
$(1,-1,1)$ & $62$ & $-34$\\
$(-1,1,1)$ & $46$ & $46$\\
$(1,-1,-1)$ & $50$ & $-46$\\
$(-1, 1, -1)$ & $34$ & $34$\\
($-1,-1,1)$ & $14$ & $14$\\
$(-1,-1,-1)$ &  $2$ & $2$\\
\end{tabular}
\end{center}

Specifically, we have 

\[\mathcal{D}(n;t)=\begin{cases}\mp [t{\frac{n\mp 98}{48}}-t^{-(\frac{n\mp 98}{48})}] & n\equiv\pm2\bmod(96)\\ 
\pm[t^{\frac{n \mp 14}{48}}-t^{-(\frac{n \mp 14}{48})} & n\equiv\pm 14\bmod{96)}\\
\pm[t^{\frac{n\mp 34}{96}}-t^{-(\frac{n\mp 34}{96})}] & n\equiv\pm34\bmod{96}\\
\mp[t^{\frac{n\mp 50}{48}}-t^{-(\frac{n\mp 50}{48})}] & n\equiv\pm46\bmod{96},\end{cases}\] which is even in $n$. When $t=\omega$ is fixed to be a $2j$-th root of unity, the map is $96j$-periodic. By \cite[Lemma $5.2$]{LM} and \cite[Lemma $2.1$]{goswamiosburn} 
\[\sum_{n \geq 0}\mathcal{D}(n;\omega)q^{\frac{n^{2}}{192}}\] is a modular form of weight $\frac{1}{2}$ with respect to $\Gamma(192j^{2})$. For our specific choice of $M$, we have that $X-Y=Y-X=\emptyset$, so $p(q)=0$.

\bibliographystyle{amsplain}
\bibliography{LLrefs}

\providecommand{\bysame}{\leavevmode\hbox to3em{\hrulefill}\thinspace}
\providecommand{\MR}{\relax\ifhmode\unskip\space\fi MR }
\providecommand{\MRhref}[2]{%
  \href{http://www.ams.org/mathscinet-getitem?mr=#1}{#2}
}
\providecommand{\href}[2]{#2}
\begin{thebibliography}{10}

\bibitem{ajk}
R.~Akhmechet, P.~K. Johnson, and V.~Krushkal, \emph{Lattice cohomology and q-series invariants of 3-manifolds}, Journal f{\"u}r die reine und angewandte Mathematik (Crelles Journal) \textbf{2023} (2023), no.~796, 269--299.

\bibitem{ajp}
R.~Akhmechet, P.K. Johnson, and S.~Park, \emph{Knot latttice homology and {$q$}-series invariants for plumbed knot complements}, Quantum Topology (2025), published online first.

\bibitem{am_resurgence}
J.E. Andersen and W.E. Misteg{\aa}rd, \emph{Resurgence analysis of quantum invariants of {S}eifert fibered homology spheres}, Journal of the London Mathematical Society \textbf{105} (2022), no.~2, 709--764.

\bibitem{bringmannquantum1}
K.~Bringmann, K.~Mahlburg, and A.~Milas, \emph{Higher depth quantum modular forms and plumbed {$3$}-manifolds}, Letters in Mathematical Physics \textbf{110} (2020), no.~10, 2675--2702.

\bibitem{bringmannquantum2}
\bysame, \emph{Quantum modular forms and plumbing graphs of {$3$}-manifolds}, Journal of Combinatorial Theory, Series A \textbf{170} (2020), 105145.

\bibitem{gukovmodularity}
M.~C.~N. Cheng, S.~Chun, F.~Ferrari, S.~Gukov, and S.~M. Harrison, \emph{3d modularity}, Journal of High Energy Physics \textbf{2019} (2019), no.~10, 1--95.

\bibitem{cheng2024quantum}
M.C.N. Cheng, I.~Coman, D.~Passaro, and G.~Sgroi, \emph{Quantum modular $\widehat{Z}^{G}$-invariants}, SIGMA. Symmetry, Integrability and Geometry: Methods and Applications \textbf{20} (2024), 018.

\bibitem{cheng20243d}
M.C.N. Cheng, P.~Coman, I.and~Kucharski, D.~Passaro, and G.~Sgroi, \emph{3d modularity revisited}, Preprint, arXiv:2403.14920 (2024).

\bibitem{fuji2021witten}
H.~Fuji, K.~Iwaki, H.~Murakami, and Y.~Terashima, \emph{Witten--{R}eshetikhin--{T}uraev function for a knot in {S}eifert manifolds}, Communications in Mathematical Physics \textbf{386} (2021), no.~1, 225--251.

\bibitem{goswamiosburn}
A.~Goswami and R.~Osburn, \emph{Quantum modularity of partial theta series with periodic coefficients}, Forum Mathematicum \textbf{33} (2021), no.~2, 451--463.

\bibitem{GKS}
S.~Gukov, L.~Katzarkov, and J.~Svoboda, \emph{{$\widehat{Z}$} and splice diagrams}, SIGMA. Symmetry, Integrability and Geometry: Methods and Applications \textbf{21} (2025), 073.

\bibitem{gukovmanolescu}
S.~Gukov and C.~Manolescu, \emph{A two-variable series for knot complements}, Quantum Topol. \textbf{12} (2021), no.~1, 1--109.

\bibitem{gppv}
S.~Gukov, D.~Pei, P.~Putrov, and C.~Vafa, \emph{{BPS} spectra and {$3$}-manifold invariants}, J. Knot Theory Ramifications \textbf{29} (2020), no.~2.

\bibitem{gpp17}
S.~Gukov, P.~Putrov, and C.~Vafa, \emph{Fivebranes and {$3$}-manifold homology}, Journal of High Energy Physics \textbf{2017} (2017), no.~7, 1--82.

\bibitem{hikami}
K.~Hikami, \emph{On the quantum invariant for the {B}rieskorn homology spheres}, International Journal of Mathematics \textbf{16} (2005), no.~06, 661--685.

\bibitem{lawrencezagier}
R.~Lawrence and D.~Zagier, \emph{Modular forms and quantum invariants of {$3$}-manifolds}, Asian Journal of Mathematics \textbf{3} (1999), no.~1, 93--108.

\bibitem{LM}
L.~Liles and E.~McSpirit, \emph{Infinite families of quantum-modular 3-manifold invariants}, Communications in Number Theory and Physics \textbf{18} (2024), 237--260.

\bibitem{murakami}
Y.~Murakami, \emph{A proof of a conjecture of {G}ukov--{P}ei--{P}utrov--{V}afa}, Communications in Mathematical Physics \textbf{405} (2024), no.~11, 274.

\bibitem{nemethi}
A.~N{\'e}methi, \emph{Lattice cohomology of normal surface singularities}, Publications of the Research Institute for Mathematical Sciences \textbf{44} (2008), no.~2, 507--543.

\bibitem{neumann}
W.~D. Neumann, \emph{A calculus for plumbing applied to the topology of complex surface singularities and degenerating complex curves}, Transactions of the American Mathematical Society \textbf{268} (1981), no.~2, 299--344.

\bibitem{web}
K.~Ono, \emph{The {W}eb of {M}odularity: {A}rithmetic of the {C}oefficients of {M}odular {F}orms and {$q$}-series}, no. 102, American Mathematical Soc., 2004.

\bibitem{RT91}
N.~Reshetikhin and V.G. Turaev, \emph{Invariants of 3-manifolds via link polynomials and quantum groups}, Inventiones mathematicae \textbf{103} (1991), no.~1, 547--597.

\bibitem{shimura}
G.~Shimura, \emph{Modular forms of half integral weight}, Modular Functions of One Variable I: Proceedings International Summer School University of Antwerp, RUCA July 17--August 3, 1972, Springer, 1973, pp.~57--74.

\bibitem{witten}
E.~Witten, \emph{Quantum field theory and the {J}ones polynomial}, Communications in Mathematical Physics \textbf{121} (1989), no.~3, 351--399.

\bibitem{zagier}
D.~Zagier, \emph{Quantum modular forms}, Quanta of maths \textbf{11} (2010), 659--675.

\end{thebibliography}

\end{document}